\documentclass[11pt,reqno]{amsart}
\usepackage{amsfonts, amssymb, color, textcomp}
\title[Non-Homogeneous Quasi-Minimizers on Metric Spaces]{Nonhomogeneous Variational Problems and Quasi-Minimizers on Metric Spaces}
\author{Jasun Gong}
\address{Jasun Gong, Department of Mathematics, University of Pittsburgh, Pittsburgh, PA 15260, United States}
\author{Juan J.\ Manfredi}
\address{Juan J.\ Manfredi, Department of Mathematics, University of Pittsburgh, Pittsburgh, PA 15260, United States}
\author{Mikko Parviainen}
\address{Mikko Parviainen, Aalto University, Institute of Mathematics, P.O.\ Box 11100, FI-00076 Aalto, Finland}
\date{Monday, 30 August 2010}
\thanks{This research was partially supported by NSF Grant DMS-1001179.}

\theoremstyle{plain}
\newtheorem{thm}{Theorem}[section]

\newtheorem{lemma}[thm]{Lemma}

\theoremstyle{definition}
\newtheorem{defn}[thm]{Definition}

\newtheorem{eg}[thm]{Example}

\numberwithin{equation}{section}

\renewcommand{\d}{\delta}	

\renewcommand{\div}{\operatorname{div}}
\newcommand{\e}{\epsilon}

\newcommand{\F}{\mathcal{F}}
\newcommand{\N}{\mathbb{N}}
\newcommand{\osc}{\operatornamewithlimits{osc}}
\newcommand{\R}{\mathbb{R}}
\newcommand{\spt}{\operatorname{spt}}

\def\Xint#1{\mathchoice
{\XXint\displaystyle\textstyle{#1}}%
{\XXint\textstyle\scriptstyle{#1}}%
{\XXint\scriptstyle\scriptscriptstyle{#1}}%
{\XXint\scriptscriptstyle\scriptscriptstyle{#1}}%
\!\int}
\def\XXint#1#2#3{{\setbox0=\hbox{$#1{#2#3}{\int}$}
\vcenter{\hbox{$#2#3$}}\kern-.5\wd0}}

\def\dashint{\Xint-}

\newcommand{\ud}{\,d}

\newcommand{\trm}{\textrm}

\newcommand{\half}{{\frac{1}{2}}}

\newcommand{\Lip}{\operatorname{Lip}}

\newcommand{\ol}{\overline}

\def\vint_#1{\mathchoice%
          {\mathop{\kern 0.2em\vrule width 0.6em height 0.69678ex depth -0.58065ex
                  \kern -0.8em \intop}\nolimits_{\kern -0.4em#1}}%
          {\mathop{\kern 0.1em\vrule width 0.5em height 0.69678ex depth -0.60387ex
                  \kern -0.6em \intop}\nolimits_{#1}}%
          {\mathop{\kern 0.1em\vrule width 0.5em height 0.69678ex
              depth -0.60387ex
                  \kern -0.6em \intop}\nolimits_{#1}}%
          {\mathop{\kern 0.1em\vrule width 0.5em height 0.69678ex depth -0.60387ex
                  \kern -0.6em \intop}\nolimits_{#1}}}
\def\vintslides_#1{\mathchoice%
          {\mathop{\kern 0.1em\vrule width 0.5em height 0.697ex depth -0.581ex
                  \kern -0.6em \intop}\nolimits_{\kern -0.4em#1}}%
          {\mathop{\kern 0.1em\vrule width 0.3em height 0.697ex depth -0.604ex
                  \kern -0.4em \intop}\nolimits_{#1}}%
          {\mathop{\kern 0.1em\vrule width 0.3em height 0.697ex depth -0.604ex
                  \kern -0.4em \intop}\nolimits_{#1}}%
          {\mathop{\kern 0.1em\vrule width 0.3em height 0.697ex depth -0.604ex
                  \kern -0.4em \intop}\nolimits_{#1}}}

\newcommand{\kint}{\vint}

\newcommand{\Om}{\Omega}
\newcommand{\eps}{\e}

\begin{document}

\maketitle


\begin{abstract}
We show that quasi-minimizers of non-homogeneous energy functionals are 
locally H\"older continuous and satisfy the Harnack inequality on metric 
measure spaces. We assume that the space is doubling and supports a 
Poincar\'e inequality. The proof is based on the De Giorgi method, 
combined with the expansion of positivity technique.
\end{abstract}
\section{Introduction}

We study minimizers of variational problems in the setting of metric 
measure spaces.  Here the energy functional is of $p$-Laplacian type.  In 
the Euclidean setting it has the form
\begin{equation} \label{eq_pdirichlet}
\int \big(|\nabla u|^p + uF\big) \,dx
\end{equation}
with $p \in (1,\infty)$, and minimizers are solutions to the 
Euler-Lagrange equation
\begin{equation} \label{eq_plaplacian}
\div(|\nabla u|^{p-2}\nabla u) \;=\; F.
\end{equation}
As our main results, we prove local H\"older continuity and a 
Harnack-type inequality for minimizers on metric-measure spaces.  In fact 
the methods are robust enough to hold for a more general class of 
functions.  Following Giaquinta and Giusti \cite{GiaquintaGiusti82}, a 
function $u \in W^{1,p}(\R^n)$ is a quasi-minimizer if there exists $K 
\geq 1$ so that
$$
\int_\Omega \big(|\nabla u|^p + uF\big) \, dx \;\leq\;
K \int_\Omega \big(|\nabla v|^p + vF\big) \, dx
$$
holds for all $\Omega \Subset  \R^n$ and for all $v \in W^{1,p}(\Omega)$ 
with $u - v \in W^{1,p}_0(\Omega)$. When $K = 1$ these are minimizers in 
the usual sense.

The usual notion of a derivative on $\R^n$ is not well-defined on an 
arbitrary metric space.  As a replacement, we use {\em upper gradients}, 
which are defined in terms of a generalized Fundamental Theorem of 
Calculus. Under mild assumptions on a metric space, the notion of upper 
gradient gives rise to analogues of Sobolev spaces and Sobolev 
inequalities, as developed in Cheeger \cite{Cheeger}, Haj{\l}asz-Koskela 
\cite{HajlaszKoskela1}, \cite{HajlaszKoskela}, Heinonen-Koskela 
\cite{HeinonenKoskela}, Semmes \cite{Semmes}, and Shanmugalingam \cite{Shan}.
Examples include spaces of non-negative Ricci curvature \cite{Buser}, 
\cite{LottVillani}, \cite{Sturm}, Carnot groups and Carnot-Carath\'{e}dory 
spaces \cite{Gromov}, boundaries of certain hyperbolic buildings 
\cite{BourdonPajot}, and self-similar fractals \cite{Laakso}.

Our approach is a variant of De Giorgi's method, which we use to prove 
local H\"older continuity for quasi-minimizers (Theorem~\ref{thm_holder}) 
as well as a Harnack-type inequality (Theorem~\ref{thm_harnack}).  The 
proof of the Harnack inequality is based on the ``expansion of 
positivity'' technique \cite{DiBenedetto1}, 
\cite{DiBenedettoGianazzaVespri} extended to metric spaces.  
This provides an alternative to the usual Krylov-Safonov covering 
technique \cite{KrylovSafonov}.
Regarding the appearance of nonhomogeneous terms $F \in L^s(\Omega)$, 
with $s > 1$, the oscillation of a quasi-minimizer $u$ is handled in a 
standard but nontrivial manner: if the norm $\|F\|_s$ is sufficiently 
large on a ball, then the oscillation of $u$ is controlled by the measure 
of the ball; otherwise it is controlled by oscillation of $u$ on larger 
concentric balls.

In the classical setting, local H\"older continuity of quasi-minimizers 
was shown by Giaquinta and Giusti in \cite{GiaquintaGiusti84} and the 
Harnack inequality by DiBenedetto and Trudinger 
\cite{DiBenedettoTrudinger}. In the case of minimizers, this follows from 
well-known techniques of De Giorgi \cite{DeGiorgi}, Nash \cite{Nash}, and 
Moser \cite{Moser1}, \cite{Moser2}; see also Lady\v{z}henskaya and 
Ural'seva \cite{LadyUral}.
We note that H\"older continuity is the most that one can expect in this 
setting: Koskela, Rajala, and Shanmugalingam \cite[p.\ 
150]{KoskelaRajalaShan} have shown that without additional geometric 
assumptions, even minimizers on {\em closed} subsets of $\R^n$ can fail 
to be locally Lipschitz continuous.

Kinnunen and Shanmugalingam studied the case of homogeneous functionals 
of $p$-Laplacian type ($F=0$) in  \cite{KinnunenShan}.  By adapting the 
De Giorgi method to metric measure spaces, they recovered local H\"older 
continuity, the Harnack inequality, and the strong maximum principle for 
quasi-minimizers.  Later Bj\"orn and Marola \cite{BjornMarola} showed 
that the Moser iteration technique can also be adapted to the metric 
setting for minimizers.

For the non-homogeneous case, Jiang \cite{Jiang} has recently shown, when 
$p=2$ and when an additional heat kernel inequality holds, that 
minimizers are locally Lipschitz continuous.
When the data $F$ is a Radon measure and when $p > 1$, 
M{\"a}k{\"a}l{\"a}inen \cite{Makalainen_nonhom} has shown that minimizers 
are H\"older continuous if and only if the measure satisfies certain 
growth conditions on balls.
Both works rely crucially on a theorem of Cheeger \cite{Cheeger}, which 
asserts that such metric measure spaces support a generalized 
differentiable structure.  We note that our techniques are independent of 
theirs.

Our methods also apply in the setting of Cheeger differentiable 
structures \cite{Cheeger}. Indeed, the results  of this paper are applied 
in the forthcoming article \cite{GongHajlasz} to prove that 
quasi-minimizers are Cheeger differentiable almost everywhere.

The paper is organized as follows.  In Section \ref{sect_prelim}, we 
review standard results in the analysis on metric spaces.  We introduce  
quasi-minimizers on metric spaces in Section \ref{sect_qmin} and prove a 
Caccioppoli-type inequality.  In Section \ref{sect_locbdd} we prove that 
quasi-minimizers are locally bounded, which motivates our study of 
certain function classes that we call {\em De Giorgi classes}, since they
are a natural generalization of the Euclidean De Giorgi classes.  We show 
H\"older continuity and a Harnack-type inequality in Sections 
\ref{sect_holder} and \ref{sect_harnack}, respectively.

\section*{Acknowledgements}

Part of this paper was written while MP visited the University of Pittsburgh and while JG visited the Aalto University School of Science and Technology; they would like to thank both universities for the hospitality.
The authors would like to thank Professor Nageswari Shanmugalingam for many helpful conversations and suggestions that led to improvements in this work.

\section{Preliminaries} \label{sect_prelim}

\subsection{Notation}

On a metric space $X$, we write $B(x,r)$ for the ball centered at $x \in X$ with radius $r$.  If no confusion arises, we write $B_r = B(x,r)$ for short.  For real-valued functions $u$, we write
$$
u_+ = \max\{u,0\} \hspace{1em} \trm{ and } \hspace{1em} u_- = -\min\{u,0\}.
$$
The oscillation of $u$ on a set $A$ is given by
$$
\mathop{\osc}_Au = \sup_Au -\inf_Au.
$$
For $h\in \R$ and a ball $B(x,r)$, we denote the super-level set of a function $u$ by
$$
A_r(h)\;=\;\{ y\in B(x,r)\,:\,u(y) > h\}.
$$
For a function $u$ in $L^p(A)$, we write the $L^p$-norm as $\|u\|_{p,A}$, or as $\|u\|_p$ if the set $A$ is the entire domain of $u$.  As usual,
the H\"older conjugate of $p \in (1,\infty)$ is given by
$$
p' =\frac{p}{p-1}.
$$

\subsection{Doubling measures}

In what follows, a {\em metric measure space} $(X,d,\mu)$ refers to a metric space $(X,d)$ equipped with a Borel measure $\mu$ on $X$.
\begin{defn} \label{defn_doubling}
Let ${c_\mu} \geq 1$.  A Borel measure $\mu$ on $X$ is said to be {\em doubling} if every ball $B(x,r)$ in $X$
has positive, finite $\mu$-measure and
\begin{equation} \label{eq_doubling}
\mu\big(B(x,2r)\big) \;\leq\; {c_\mu} \, \mu\big( B(x,r) \big).
\end{equation}
\end{defn}

The doubling exponent $Q := \log_2({c_\mu})$ plays the analogous role of 
dimension on metric measure spaces.  In particular, for $p \in (1,Q)$ we 
define the (Sobolev) conjugate exponents as
$$
p^* \;=\; \frac{Qp}{Q-p}.
$$
For {\em connected} metric spaces, the doubling property 
\eqref{eq_doubling} implies that locally the $\mu$-measures of balls are 
controlled by powers of their radii.  The lemma below is well-known.  
The first item is \cite[Lemma 4.7]{Hajlasz} and for completeness, we prove 
the second item.

\begin{lemma} \label{lemma_lowerahlforsreg}
Let $X$ be a metric space, let $\mu$ a doubling measure on $X$, and let $Q$ be the doubling exponent of $\mu$.  For each ball $B_0 = B(x_0,r_0)$ in $X$, with $0 < r_0 < \infty$,
\begin{enumerate}
\item there exists $c = c({c_\mu},B_0) > 0$ so that for all $x \in B_0$ and all $r \in (0,r_0)$ we have the inequality
$$
c\Big( \frac{r}{r_0} \Big)^Q \;\leq\;
\frac{\mu(B(x,r))}{\mu(B_0)},
$$
\item if $X$ is path-connected, then there exist constants $c = c({c_\mu},B_0) > 0$ and  $Q' = Q'({c_\mu},B_0) > 0$ so that for all $x \in B_0$ and all $r \in (0,r_0)$ we have the inequality
$$
\frac{\mu(B(x,r))}{\mu(B_0)} \;\leq\; c\Big( \frac{r}{r_0} \Big)^{Q'}.
$$
\end{enumerate}
\end{lemma}

\begin{proof}[Proof of (2)]
Fix a ball $B = B(x,r)$ in $X$.  Since $X$ is path-connected, the sphere $\partial B(x,\frac{3}{7}r)$ is nonempty, so let $z \in  B$ be a point in $\partial B(x,\frac{3}{7}r)$ and consider the ball $B' = B(z,\frac{4}{7}r)$.  Clearly $B' \subset B$ and $\frac{1}{7}B \subset B' \setminus \frac{1}{2}B'$, which imply
\[
\begin{split}
\frac{1}{c_\mu} \mu(B') \leq
\mu\Big(\frac{1}{2}B'\Big)
\end{split}
\]
and as a result,
\[
\begin{split}
\mu\Big(\frac{1}{7}B\Big) \;\leq\;
\mu(B')-\mu\Big(\half B'\Big)
\;\leq
(1-\frac{1}{c_\mu}) \, \mu(B')
\;\leq (1-\frac{1}{c_\mu}) \, \mu(B).
\end{split}
\]
Now an iteration gives us
\[
\mu(B(x,r))\leq (1-\frac{1}{c_\mu})^i \mu(B(x,r_0)),
\]
where $7^{i}r \leq r_0\leq 7^{i+1} r$.
A substitution for $i$ and an application of the doubling condition finishes the proof.
\end{proof}

\subsection{Newtonian-Sobolev spaces and Poincar\'e inequalities} \label{sect_standhyp}

To define Sobolev spaces on metric measure spaces, we use  {\em weak 
upper gradients}, which are defined in terms of line integrals and a 
generalized Fundamental Theorem of Calculus.  This in turn requires a 
tool to measure the size of families of rectifiable curves.

To this end let $\Gamma$ be a collection of non-constant rectifiable curves on $X$.  For $p \geq 1$, the {\em $p$-modulus} of $\Gamma$ is defined as
$$
{\rm mod}_p(\Gamma) \;:=\; \inf_\rho \int_X \rho^p \,d\mu
$$
where $\rho : X \to [0,\infty]$ is any Borel function that satisfies 
$\int_\gamma \rho \,dx \geq 1$.  (We follow the convention that 
$\inf\emptyset = \infty$.)
The $p$-modulus is an outer measure on $\mathcal{M}$, the family of all rectifiable curves on $X$; for details, see for example \cite[Chap 7]{Heinonen}.

\begin{defn} \label{defn_wkupgrad}
For a function $u : X \to \R$, we say that a Borel function $g : X \to 
[0,\infty]$ is an {\em upper gradient} of $u$ if the inequality
\begin{equation} \label{eq_upgrad}
|u(\gamma(b)) - u(\gamma(a))| \;\leq\; \int_\gamma g \,ds
\end{equation}
holds for every rectifiable curve $\gamma : [a,b] \to X$ under its 
arc-length parametrization.

We say that $g : X \to [0,\infty]$ is a {\em weak upper gradient} of $u$ if Equation \eqref{eq_upgrad} holds for $p$-modulus a.e.\ curve $\gamma \in \mathcal{M}$ --- that is, if $\Gamma$ is the subcollection of curves in $\mathcal{M}$ for which Equation \eqref{eq_upgrad} fails, then ${\rm mod}_p(\Gamma) = 0$.
\end{defn}

\begin{eg}
Let $u: X \to \R$ be a {\em Lipschitz  function} -- that is, it satisfies
$$
\Lip(u) \;:=\; \sup\left\{ \frac{|f(x)-f(y)|}{d(x,y)} \,:\, x, y \in X,\; x \neq y \right\} \:<\; \infty
$$
then $\Lip(u)$, called the {\em Lipschitz constant} of $u$, is an upper gradient of $u$.
\end{eg}

We now define an analogue of the Sobolev space $W^{1,p}(\R^n)$ on metric spaces.

\begin{defn} \label{defn_newtonianspace}
Let $p \geq 1$.  We say that a function $u : X \to \R$ lies in $\tilde{N}^{1,p}(X)$ if and only if $u \in L^p(X)$ and the quantity
$$
\|u\|_{1,p} \;:=\;
\|u\|_p \;+\;
\inf_g \|g\|_p
$$
is finite, where the infimum is taken over all weak upper gradients $g$ of $u$.

The {\em Newtonian space} $N^{1,p}(X)$ consists of equivalence classes of functions in $\tilde{N}^{1,p}(X)$.  Here, two functions $u, v \in \tilde{N}^{1,p}(X)$ are {\em equivalent} if $u = v$ $\mu$-a.e.
\end{defn}

We note that $\| \cdot \|_{1,p}$ is a norm and $N^{1,p}(X)$ is a Banach space with respect to this norm \cite[Thm 3.7]{Shan}.  Moreover, for each $u \in N^{1,p}(X)$, there exists a weak upper gradient $g_u$ so that the infimum in $\|u\|_{1,p}$ is attained \cite[Thm 7.16]{Hajlasz}.
We call $g_u$ the {\em minimal upper gradient} of $u$, which is uniquely determined $\mu$-a.e.

A Leibniz product rule holds for upper gradients \cite[Lemma 2.14]{Shan2}.

\begin{lemma} \label{lemma_leibniz}
If $u \in N^{1,p}(X)$ and if $f : X \to \R$ is a bounded Lipschitz function, then $u \cdot f \in N^{1,p}(X)$ and its minimal upper gradient satisfy
$$
g_{u \cdot f} \;\leq\; g_u \, |f| + |u| \, \Lip(f).
$$
\end{lemma}

We now formulate Poincar\'e inequalities in terms of weak upper gradients.  Together with the doubling property \eqref{eq_doubling}, such inequalities determine a rich theory of first-order calculus on the underlying spaces.

\begin{defn}
\label{defn_PI}
We say that a metric measure space $(X,d,\mu)$ supports a {\em (weak)\footnote{Here ``weak'' refers to the possibility that $\Lambda > 1$.} $(1,p)$-Poincar\'e inequality} if there exist $C \geq 0$, $\Lambda \geq 1$ so that
\begin{equation} \label{eq_PI}
\dashint_B |u - u_B| \,d\mu \;\leq\;
C \, r \, \Big( \dashint_{\Lambda B} g_u^p \,d\mu \Big)^\frac{1}{p}
\end{equation}
holds for all $u \in N^{1,p}_{\rm loc}(X)$ and for all balls $B$ in $X$.
\end{defn}

\theoremstyle{plain}
\newtheorem{stand}[thm]{Standing Hypotheses}
\begin{stand} \label{standhyp}
We will always assume that a metric space $(X,d)$ is equipped with a doubling measure $\mu$ and supports a (weak) $(1,p)$-Poincar\'e inequality, for some $p \in (1,Q)$; that is, Equations \eqref{eq_doubling} and \eqref{eq_PI} hold under some choice of constants ${c_\mu}, \Lambda \geq 1$ and $C > 0$.  Our main results are local in nature, so for simplicity we will work with bounded domains $\Om$ in $X$.

\end{stand}

Note that, if $(X,d,\mu)$ satisfies Standing Hypotheses \ref{standhyp}, then for $q > Q$ an analogue of Morrey's inequality holds \cite[Thm 5.1]{HajlaszKoskela}, so functions in $N^{1,q}(X)$ are already locally H\"older continuous in this case.
Note also that such spaces $X$ are {\em $c$-quasiconvex}; that is, every pair of points $x,y \in X$ can be joined by a curve in $X$ whose length is at most $c \cdot d(x,y)$.  Here $c > 0$ depends only on the parameters of the hypotheses, see \cite{DavidSemmes} and also \cite[Sect 17]{Cheeger}.
In particular, such spaces are path-connected, so the estimates of Lemma \ref{lemma_lowerahlforsreg} apply to balls in $X$.

In the same setting, Keith and Zhong \cite[Thm 1.0.1]{KeithZhong} showed that a (weak) $(1,p)$-Poincar\'e inequality for Lipschitz functions on $X$ is an open-ended condition in the
exponent $p$.  Moreover, for such spaces $X$, it is known that Lipschitz functions are dense in $N^{1,p}(X)$ \cite{Shan}.  This leads to the following theorem.

\begin{thm}[Keith-Zhong] \label{thm_openended}
If $(X,d,\mu)$ supports a (weak) $(1,p)$-Poincar\'e inequality, then there exists $\e > 0$ so that for all $q > p -\e$,
there exist $C > 0$ and $\Lambda \geq 1$ so that, for all $u \in N^{1,p}_{\rm loc}(X)$, we have
$$
\dashint_B |u - u_B| \,d\mu \;\leq\;
C \, r \, \Big( \dashint_{\Lambda B} g_u^q \,d\mu \Big)^\frac{1}{q}
$$
\end{thm}

As a consequence, we recover a version of the Sobolev embedding theorem; see 
\cite[Eq (2.11)]{KinnunenShan}.

\begin{lemma} \label{lemma_SPI}
Let $(X,d,\mu)$ be a metric measure space that supports a $(1,p)$-Poincar\'e inequality and where $\mu$ is doubling.
For $p < Q$, for $\e > 0$ as in Theorem \ref{thm_openended}, and for $p-\e < q < p$, there exist $c > 0$ and $\Lambda \geq 1$ so that the inequality
$$
\Big( \dashint_{B} |u|^t \,d\mu \Big)^\frac{1}{t} \;\leq\;
c \, r\, \Big( \dashint_{\Lambda B} g_u^q \, d\mu \Big)^\frac{1}{q}
$$
holds for all balls $B$ with $3B \subset X$, all $t \in [1,q^*]$, and all $u \in N^{1,p}_0(B)$.
\end{lemma}

\section{Quasi-minimizers and Caccioppoli-type Inequalities} \label{sect_qmin}

Using the notions of  (minimal) upper gradients, 
we now define quasi-minimizers as in\cite[Chap 6]{Giusti}.
Let $\F_0 : \Omega \times \R \times \R \to \R$ and $\Omega' \Subset  \Omega$, and consider the induced ``$p$-energy''
functional on  $N^{1,p}_{\rm loc}(\Omega)$ given by
\begin{equation} \label{eq_fnl}
\F(u;\Omega') \;:=\; \int_{\Omega'} \F_0\big(x,u(x),g_u(x)\big) \,d\mu(x).
\end{equation}

\newtheorem{struct}[thm]{Structure Conditions}
\begin{struct} \label{structcond}
Here and in later sections, we will assume that $\F_0$ satisfies the inequalities
\begin{equation} \label{eq_structcond}
|z|^p - f_1(x) \, |u|^p - f_0(x) \;\leq\;
\F_0(x,u,z) \;\leq\;
L|z|^p + f_1(x) \, |u|^p + f_0(x)
\end{equation}
for $L \geq 1$ and $f_0, f_1 \in L^s(\Omega)$, where $s > \frac{Q}{p} > 1$.
Moreover, we write
$$
\d \;:=\; \frac{p}{Q} - \frac{1}{s}.
$$
\end{struct}
Note that the $p$-Laplacian functional from \eqref{eq_pdirichlet} also satisfies Structure Conditions \ref{structcond}.  Indeed, from the elementary inequality
$t \leq t^p +1$
for $t \geq 0$, we see that  \eqref{eq_structcond} follows from the choices $f_0=f_1=|F|$.

\begin{defn} \label{defn_qmin}
We say that $u \in N^{1,p}_{\rm loc}(\Omega)$ is a {\em $K$-quasi-minimizer} if there exists $K \geq 1$ so that the inequality
\begin{equation} \label{eq_qmin}
\mathcal{F}\big(u;\Omega' \cap \{u \neq v\}\big) \;\leq\;
K \, \mathcal{F}\big(v;\Omega' \cap \{ u \neq v \}\big)
\end{equation}
holds for all $v \in N^{1,p}(\Om')$ with $u - v \in N^{1,p}_0(\Om')$, where $\Omega'\Subset \Omega$.  If $K = 1$, then $u$ is
called a {\em minimizer} of $\mathcal{F}$.
\end{defn}

As in the case of Euclidean spaces, quasi-minimizers satisfy a Caccioppoli-type inequality.  Again, we assume that Standing Hypotheses \ref{standhyp} and Structure Conditions \ref{structcond} are in force, and denote level sets by
$$
A_r(h) := \{x \in B_r : u(x) > h \} \,\textrm{ and }\,
D_r(h) := \{x \in B_r : u(x) < h \}.
$$

\begin{lemma}[Caccioppoli inequality] \label{lemma_caccioppoli}
Let $K \geq 1$. 
For each $K$-quasi-minimizer $u \in N^{1,p}_{\rm loc}(\Omega)$ and for each $h \in \R$, we have
\begin{align}
\int_{B(x,r)} g_{(u-h)_+}^p  d\mu \; \leq \;
\frac{C}{(R-r)^{p}} & \int_{B(x,R)} (u-h)_+^p \,d\mu  \notag \\ &  \;+\;
\Big(\|f_0\|_s + 2|h|^p\|f_1\|_s\Big) \, \mu\big(A_R(h)\big)^{1 - \frac{1}{s}}\notag
\end{align}
for $B(x,R) \subset \Omega$,  $0 < r < R \leq R_0$.
Here  $C = C(p,Q,K) >0$ and $R_0 = R_0(p,Q,f_1) > 0$.
\end{lemma}

The proof is in two parts. In Part (1) one uses the quasi-minimizing property to compare the $p$-energies between different balls.  In Part (2) we use a variant of Widman's hole filling argument \cite{Widman} and an iteration in order to estimate the upper gradient by the function and its level sets.

\begin{proof}
{\bf Part (1): Energy Bounds}.
Let $B_r := B(x,r)$ and $B_R := B(x,R)$. Let $\eta : X \to \R$ be a Lipschitz function so that
$\spt{\eta} \subset\bar{B}_R$, as well as $\eta|_{B(x,r)} = 1$ and $g_\eta \leq C(R-r)^{-1}$.
Putting
$$
v \;:=\; u - (u-h)_+\eta,
$$
it follows that $u-v \in N^{1,p}_0(X)$ and $\spt(u-v) = \overline{A_R(h)}$. 
By the quasi-minimizing property \eqref{eq_qmin} and the structure conditions
\eqref{eq_structcond}, we obtain
\begin{equation} \label{eq_applyqmin}
\begin{split}
\int_{A_R(h)} \big( g_u^p - f_1|u|^p - f_0 \big)\,d\mu &\;\leq\;
\F(u;A_R(h)) \;\leq\;
K \, \F(v;A_R(h)) \\ &\;\leq\;
K \int_{A_R(h)} \big( g_v^p + f_1|v|^p + f_0 \big)\,d\mu.
\end{split}
\end{equation}
For points in $A_R(h)$, we rewrite the functions $u$ and $v$ as
\begin{eqnarray*}
v \;=\;
u - \eta(u-h)_+ &=&
(1-\eta)(u-h) + h \\ &=&
(1-\eta)u + \eta h \\
u \;=\;
(1-\eta)u + \eta u &=&
(1-\eta)u + \eta h + \eta (u-h),
\end{eqnarray*}
in order to obtain the estimates
\begin{equation}
\label{eq_temp2}
g_v \leq
\frac{(u-h)_+}{R-r} \,+\, (1-\eta)g_u
\end{equation}
and
\begin{equation}\label{eq_temp3}
|u|^p + |v|^p \leq
C \Big[ (1-\eta)^pu^p + \eta^ph^p  + \eta^p(u-h)_+^p\Big].
\end{equation}
Adding the term $\int_{B_R} f_1 (2|u|^p+f_0)\,d\mu$ to both sides of \eqref{eq_applyqmin}, it follows from inequalities \eqref{eq_temp2} and \eqref{eq_temp3} that
\begin{equation}
\label{eq_initial_estimate}
\begin{split}
\int_{A_R(h)} \big( g_u^p + f_1|u|^p \big)\,d\mu \;\leq\;&
2K \int_{A_R(h)} \big(
g_v^p + f_1\big(|u|^p+|v|^p\big) + f_0
\big)\,d\mu \\
\leq\;&
C\int_{A_R(h)} \big(
g_v^p + 2f_1\big[
(1-\eta)^pu^p + \eta^p|h|^p \big] + f_0
\big)\,d\mu \\ & +\;
C\int_{B_R} f_1\eta^p(u-h)_+^p \,d\mu
\end{split}
\end{equation}
To estimate the rightmost term, we use H\"older's inequality, the Sobolev 
inequality (Lemma \ref{lemma_SPI}), and the Leibniz rule (Lemma 
\ref{lemma_leibniz}) so that
\begin{eqnarray*}
\int_{B_R} f_1 (\eta(u-h)_+)^p \,d\mu &\leq&
\|f_1\|_{\frac{Q}{p},B_R}
\left( \int_{B_R} \eta^p (u-h)_+^p\,d\mu
\right)^{p^*/p} \\ &\leq&
\|f_1\|_{\frac{Q}{p},B_R}
\int_{B_R} g_{\eta(u-h)_+}^p\,d\mu \\ &\leq&
\|f_1\|_{\frac{Q}{p},B_R}
\int_{B_R} \Big(
\eta^p g_{(u-h)_+}^p \,+\, g_\eta^p(u-h)_+^p
\Big) \,d\mu \\ &\leq&
\|f_1\|_{\frac{Q}{p},B_R} \Big[
\int_{A_r(h)} g_u^p \,d\mu \;+\;
\int_{A_R(h)} \frac{(u-h)_+^p}{(R-r)^p} \,d\mu
\Big].
\end{eqnarray*}
This together with \eqref{eq_initial_estimate} implies
\[
\begin{split}
\int_{A_R(h)} &\big( g_u^p + f_1|u|^p \big)\,d\mu \\
\leq\;&
C \int_{A_R(h)} \Big(
\frac{(u-h)_+^p}{(R-r)^p} \,+\, (1-\eta)^p\big(g_u^p + f_1|u|^p\big)
+ f_1|h|^p + f_0
\Big)\,d\mu \\
& \;+\;
C\|f_1\|_{\frac{Q}{p},B_R} \Big[
\int_{A_r(h)} g_u^p \,d\mu \;+\;
\int_{A_R(h)} \frac{(u-h)_+^p}{(R-r)^p} \,d\mu
\Big]
\end{split}
\]
Choosing $R_0 > 0$ small enough so that $C \|f_1\|_{\frac{Q}{p}, B(x,R_0)} < \half$,
we obtain
\begin{eqnarray*}
\int_{A_r(h)} \big( g_u^p + f_1|u|^p \big)\,d\mu &\leq&
C \int_{A_R(h) \setminus A_r(h)} \big(g_u^p + f_1|u|^p\big) \, d\mu \,+\,
\frac{1}{2} \int_{A_r(h)} g_u^p \,d\mu \\ & & \;+\;
C\int_{A_R(h)} \Big[
\frac{(u-h)_+^p}{(R-r)^p} \,+\, 2f_1|h|^p + f_0
\Big]\,d\mu.
\end{eqnarray*}

\noindent{\bf Part (2): Hole filling}.
Adding
$(C-\frac{1}{2})\int_{A_r(h)} \big( g_u^p + f_1|u|^p\big) \,d\mu$
to both sides and dividing by $C+\frac{1}{2}$, we obtain, for $\theta :=
\frac{2C}{2C+1}$, the inequality
\begin{eqnarray*}
\int_{A_r(h)} \big( g_u^p + f_1|u|^p \big)\,d\mu &\leq&
\theta \int_{A_R(h)} \big(g_u^p + f_1|u|^p\big) \, d\mu \\ & & \,+\,
\int_{A_R(h)} \Big[
\frac{(u-h)_+^p}{(R-r)^p} + 2f_1|h|^p + f_0
\Big] \,d\mu.
\end{eqnarray*}
Next we iterate this equation, under the choice of radii
\begin{eqnarray*}
r_0&=&r, \\
r_{i}-r_{i-1}&=&(1-\lambda)\lambda^i  (R-r),\
\trm{ for } i=1,2,\ldots, \trm{ where } \lambda^p \in (\theta,1),
\end{eqnarray*}
so the previous estimate becomes
\begin{equation} \label{eq_iterholefill}
\begin{split}
\int_{A_r(h)} \big( g_u^p + f_1|u|^p \big)\,d\mu \;\leq\;&
\theta^k \int_{A_{r_i}} \big(g_u^p + f_1|u|^p\big) \, d\mu \\ &\;+\;
\sum_{i=0}^k  \theta^{i} \int_{A_{r_i}}
\frac{(u-h)_+^p}{(r_i-r_{i-1})^p} +
\big(2f_1|h|^p + f_0\big) \Big] \,d\mu.
\end{split}
\end{equation}
Passing to a limit, as $k\to \infty$, gives
\[
\begin{split}
\int_{A_r(h)} \big( g_u^p + f_1|u|^p \big)\,d\mu \;\leq\;&
C \int_{A_R(h)} \Big[
\frac{(u-h)_+^p}{(R-r)^p} + 2f_1|h|^p + f_0
\Big] \,d\mu \\ \leq\;&
C \int_{A_R(h)} \frac{(u-h)_+^p}{(R-r)^p} \,d\mu \;+\;
C \int_{A_R(h)} \big(2f_1|h|^p + f_0 \big) \,d\mu.
\end{split}
\]
and the lemma follows, from applying H\"older's
inequality to the last term:
$$
\hspace{3em}
\int_{B_R} \big(2f_1|h|^p + f_0\big) \,d\mu \;\leq\;
\big( 2\|f_1\|_{s,B_R} |h|^p + \|f_0\|_{s,B_R} \big) \,
\mu(A_R(h))^\frac{s-1}{s}.
\hspace{3em}\qedhere
$$
\end{proof}

The proof above remains valid with $-u$, $-v$, and $-h$ in place of $u$,
$v$, and $h$, respectively.  From this we conclude that, for each $h \in
\R$, the inequality
$$
\int_{B_r} g_{(u-h)_-}^p d\mu \;\leq\;
C \int_{B_R} \frac{(u-h)_-^p}{(R - r)^p}  \,d\mu \,+\,
\big(\|f_0\|_s + 2|h|^p\|f_1\|_s\big) \,
\mu\big(D_R(h)\big)^{1 - \frac{1}{s}}
$$
holds for quasi-minimizers $u \in N^{1,p}(\Om)$, with the same constants as before.

\section{Local Boundedness and De Giorgi Classes} \label{sect_locbdd}

\subsection{Initial Estimates}

As a first step, we show that every quasi-minimizer has an a.e.\ representative that is locally bounded; see \cite[Thm 1]{DiBenedettoTrudinger} and \cite[Thm 10.2.1]{DiBenedetto} for the case of $\R^n$ and \cite[Thm 4.9]{KinnunenShan} for the case $F=0$ on metric spaces.
We begin with a well-known iteration lemma, see for example \cite{Giusti}.

\begin{lemma} \label{lemma_fastconv}
Let $b > 1$ and $\sigma, C > 0$ be given.  If $\{Y_n\}_{n=0}^\infty$ is a sequence in $[0,\infty)$ whose terms satisfy, for $n=0,1,\ldots$, the inequalities
$$
Y_{n+1} \;\leq\; C b^n Y_n^{1+\sigma} \; \textrm{ and } \;
Y_0 \;\leq\; b^{-1/\sigma^2}C^{-1/\sigma}
$$
then $Y_n\leq b^{-n/\sigma} Y_0$ and in particular
\[
\begin{split}
Y_n \to 0\quad \trm{as}\quad n \to \infty.
\end{split}
\]
\end{lemma}

Next we prove the local boundedness for quasi-minimizers.
Below, recall that $Q' > 0$ refers to the exponent in Part (2) of Lemma \ref{lemma_lowerahlforsreg}, $\Lambda \geq 1$ refers again to the parameter in Lemma \ref{lemma_SPI}, and $\d := \frac{p}{Q} - \frac{1}{s}$.

\begin{lemma} \label{lemma_weakharnack}
There exists $C = C(p,Q,K,s) \geq 0$ so that the inequalities
\begin{eqnarray*}
\sup_{B(x,R/2)} u &\leq&
C\Big[ \dashint_{B(x,R)} u_+^p \,d\mu \Big]^{1/p} +\; 2\gamma R^{(Q'\d)/p} \\
\inf_{B(x, R/2)} u &\geq&
-C\Big[ \dashint_{B(x,R)} u_-^p \,d\mu \Big]^{1/p} -\; 2\gamma R^{(Q'\d)/p}
\end{eqnarray*}
hold for each quasi-minimizer $u \in N^{1,p}_{\rm loc}(\Omega)$ and all balls $B(x,R)$ in $\Omega$ with sufficiently small $\mu$-measure and  with $B(x,2\Lambda R) \subset \Omega$.  Here
$$
\gamma \;:=\; \max\{\|f_0\|_s, 2\|f_1\|_s\}.
$$
\end{lemma}

The proof below follows a technical iteration argument that is, to some 
extent, standard.  We will also use similar arguments to prove other 
results in this section.  For the sake of exposition we divide it into two 
steps: (1) By using Caccioppoli's and Sobolev's inequalities, we derive a 
level set inequality with higher level set on the right hand side, and (2) 
we iterate the estimate.

\begin{proof}
{\bf Part (1):\ Level set inequality}.
Since $-u$ is also a quasi-minimizer, the inequality for the infimum follows easily from the inequality for the supremum of $-u$, so we prove the supremum inequality.

Let $r > 0$ with $R/2 < r < R$, let be $\eta$ a cut-off function such that $\spt \eta\subset \ol B_{R}$, $\eta=1$ in $B_r$, $g_\eta\leq C/(R-r) $ and let $k>0$.  By using H\"older's and Sobolev's inequalities, we obtain
\[
\begin{split}
\dashint_{B_r} (u-k)_+^p \ud \mu
&\leq\; \left[ \frac{\mu(A_r(k))}{\mu(B_r)}\right]^\frac{p}{Q}\left[\dashint_{B_r} (u-k)_+^{p^*} \ud \mu\right]^\frac{p}{p^*} \\
&\leq\; \left[ \frac{\mu(A_r(k))}{\mu(B_r)}\right]^\frac{p}{Q} \left[\dashint_{B_{(R+r)/2}} (\eta(u-k)_+)^{p^*} \ud \mu\right]^\frac{p}{p^*} \\
&\leq\; C \left[ \frac{\mu(A_r(k))}{\mu(B_r)}\right]^\frac{p}{Q} R^p \dashint_{B_{(R+r)/2}} g_{\eta(u-k)_+}^{p} \ud \mu.
\end{split}
\]
By the Leibniz rule (Lemma~\ref{lemma_leibniz}) and the Caccioppoli inequality (Lemma \ref{lemma_caccioppoli}), we have for $0<h<k$ that
\begin{equation} \label{eq_level1}
\begin{split}
\dashint_{B_r} &(u-k)_+^p \ud \mu \\
\leq\;& C \left[ \frac{\mu(A_r(k))}{\mu(B_r)}\right]^\frac{p}{Q} R^p \left[
\dashint_{B_{(R+r)/2}} \big(
\eta^p g_{(u-k)_+}^{p} + g_\eta^p(u-k)_+^p
\big) \,d\mu
\right] \\
\leq\;& C\left[ \frac{\mu(A_r(k))}{\mu(B_r)}\right]^\frac{p}{Q}R^p
 \left[
\dashint_{B_R} \frac{(u-h)_+^p}{(R-r)^p} \ud \mu +
\gamma (1+k^p)\frac{\mu(A_R(k))^{1-1/s}}{\mu(B_r)}
\right],
\end{split}
\end{equation}
holds, where $\gamma := \max\{\|f_0\|_s, 2\|f_1\|_s\}$.
By Lemma \ref{lemma_lowerahlforsreg}, we may assume that  
\[
\begin{split}
R^p\leq C \mu(B_R)^\frac{p}{Q}
\end{split}
\]
and recalling that $1-1/s=1-\frac{p}{Q}+\delta$, we obtain
\[
\begin{split}
\left[ \frac{\mu(A_r(k))}{\mu(B_r)}\right]^\frac{p}{Q}R^p\mu(A_R(k))^{1-\frac{1}{s}}
=\;&\left[ \frac{\mu(A_r(k))}{\mu(B_r)}\right]^\frac{p}{Q}R^p\mu(A_R(k))^{1-\frac{p}{Q}+\delta}\\
=\;&C \left[ \frac{\mu(A_r(k))}{\mu(B_r)} \frac{\mu(B_R)}{\mu(A_R(k))}\right]^\frac{p}{Q}\mu(A_{R}(k))^{1+\delta}\\
\leq&\; C \mu(A_R(k))^{1+\delta}.
\end{split}
\]
From this, \eqref{eq_level1}, and the elementary estimate for $h>k$, we have
\[
\begin{split}
\int_{B_R} (u-h)_+^p \,d\mu \;>\;
\int_{A_R(k)} (u-h)_+^p \,d\mu \;\geq\;
(k-h)^p \mu(A_R(k)),
\end{split}
\]
equation \eqref{eq_level1} becomes
\[
\begin{split}
\dashint_{B_r} (u-k)_+^p \ud \mu
\leq\;&\; C\left[ \frac{\mu(A_r(k))}{\mu(B_r)}\right]^\frac{p}{Q} \frac{R^p}{(R-r)^p}
\kint_{B_R} (u-h)_+^{p} \ud \mu \\
&\;+\;
\frac{C\gamma (1+k^p)\mu(B_{R})^\d}{(k-h)^{p(1+\delta)}}\left(\dashint_{B_R} (u-h)_+^p \ud\mu\right)^{1+\delta}
\end{split}
\]
and dividing by $(k-h)^p$, we obtain
\begin{equation} \label{eq_pre_iter}
\begin{split}
\dashint_{B_r} \frac{(u-k)_+^p}{(k-h)^p} \ud \mu
\;\leq&\; C \frac{R^p}{(R-r)^p}
\left(\kint_{B_R} \frac{(u-h)_+^p}{(k-h)^p} \ud \mu \right)^{1+\frac{p}{Q}} \\
&\;+\;
\frac{C\gamma (1+k^p)\mu(B_{R})^\d}{(k-h)^p}\left(\dashint_{B_R} \frac{(u-h)_+^p}{(k-h)^p} \ud\mu\right)^{1+\delta}
\end{split}
\end{equation}

\noindent {\bf Part (2):\ Iteration}.
We now iterate the previous inequality with $h$ and $k$ replaced by $k_n$ and $k_{n+1}$, respectively, and where
\[
\begin{split}
k_n=d(1-2^{-n})
\end{split}
\]
and where $d>0$ is a parameter to be chosen later.  Similarly, the balls $B_r$ and $B_R$ are replaced by $B_n$ and $B_{n+1}$, respectively, where
$$
B_n \;=\; B(x,r_n), \trm{ for } r_n=\frac{R}{2}(1+2^{-n}).
$$
For the sequence of integrals
$$
Y_n \;:=\; d^{-p}\dashint_{B_n} (u-k_n)_+^p \,d\mu,
$$
equation \eqref{eq_pre_iter} then becomes
\begin{equation}
\label{eq_fundamental_estimate}
\begin{split}
Y_{n+1}
\leq&C [2^{pn}
Y_{n}]^{1+\frac{p}{Q}} +
\frac{C\gamma (1+k_{n+1}^p)\mu(B_n)^\d }{d^p}[2^{pn}Y_n]^{1+\delta}
\end{split}
\end{equation}
We may estimate the rightmost term, by means of the inequality
\[
\begin{split}
\gamma (1+k_{n+1}^p)\mu(B_0)^\d \;\leq\;
\gamma (1+d^p)\mu(B_0)^\d \;\leq\; d^p.
\end{split}
\]
  Indeed, this follows from choosing $B_0$ sufficiently small so that $\gamma \mu(B_0)^\d < 1/2$, as well as $d$ sufficiently large so that
$$
\frac{\gamma \mu(B_0)^\d}{1-\gamma \mu(B_0)^\d} \;\leq\; 2\gamma \mu(B_0)^\d \;\leq\; d^p.
$$
Without loss of generality, we may assume that $Y_n \leq 1$.
The above estimate and \eqref{eq_fundamental_estimate} imply that
\begin{equation} \label{eq_iter}
\begin{split}
Y_{n+1} &\;\leq\;
C\Big( [2^{np}Y_n]^{1+\frac{p}{Q}} \,+\, [2^{np} Y_n]^{1+\d} \Big) \;\leq\;
C2^{np(1+\sigma')}Y_n^{1+\sigma} \;=:\;
\hat Cb^nY_n^{1+\sigma}
\end{split}
\end{equation}
where, as a shorthand, we write
\begin{equation} \label{eq_iter-param}
\sigma \,:=\, \min\big\{ \frac{p}{Q}, \d \big\}, \hspace{.1in}
\sigma' \,:=\, \max\big\{ \frac{p}{Q}, \d \big\}, \hspace{.1in}
b \;:=\; 2^{p(1+\sigma')}.
\end{equation}
Choosing $d$ larger if necessary, so that the inequality
$$
Y_0 \;=\;
d^{-p} \dashint_{B_0} (u - k_0)_+^p \,d\mu \;=\;
d^{-p} \dashint_{B(x,R)} u_+^p \,d\mu \;\leq\;
\min\{ \hat C^{-1/\sigma}b^{-1/\sigma^2},1\}
$$
holds, we invoke Iteration Lemma \ref{lemma_fastconv} and conclude that
$$
0 \;=\;
\lim_{n \to \infty} Y_n \;=\;
\lim_{n \to \infty} \dashint_{B_n} (u-k_n)_+^p \,d\mu \;=\;
\dashint_{B(x,R/2)} (u-d)_+^p \,d\mu.
$$
As a result, $u \leq d$ holds a.e.\ on $B(x,R/2)$.  In particular, for the choice
\begin{equation} \label{eq_iter-cond}
d \;:=\; \max\left\{
2\gamma\,\mu(B(x,R))^\frac{\d}{p},\,
\Big( b^{1/\sigma^2}\hat C^{1/\sigma} \dashint_{B(x,R)} u_+^p \,d\mu \Big)^\frac{1}{p}
\right\}
\end{equation}
we obtain the inequality
$$
\sup_{B(x,R/2)} u \;\leq \; d \;\leq\;
C'\Big[ \dashint_{B(x,R)} u_+^p \,d\mu \Big]^\frac{1}{p} \;+\; 2\gamma\mu(B(x,R))^\frac{\d}{p}
$$
where $C' := \big(b^{1/\sigma^2}\hat C^{1/\sigma}\big)^\frac{1}{p}$.  The lemma then follows from Lemma \ref{lemma_lowerahlforsreg}.
\end{proof}

\subsection{De Giorgi classes}

In his study of elliptic PDE, De Giorgi observed that the validity of a Caccioppoli-type inequality for solutions implies regularity properties of the same solutions.  We will therefore focus on classes of functions, called {\em De Giorgi classes}, that satisfy such inequalities.  Since quasi-minimizers are a subset of these functions, we will not refer explicitly to the quasi-minimizing property \eqref{eq_qmin} in the sequel.

We first modify the Caccioppoli inequality to obtain simpler nonhomogeneous terms. To this end, fix a ball $B$  and consider the parameters
$$
M \;:=\; \max\left\{ \Big(\dashint_{B} |u|^p\,d\mu\Big)^{1/p} \,+\; c\, r^{(Q'\d)/p} \right\}
\; \textrm{ and } \;
g_0 \;:=\; f_0 + Mf_1
$$
which are well-defined, by Lemma~\ref{lemma_weakharnack}, and where $f_0$ and $f_1$ are from the structure conditions \eqref{eq_structcond}. Observe that $u$ is also a quasi-minimizer of the restricted functional $\mathcal{G}[v] := \mathcal{F}[v|_{B}]$.  Since $\mathcal{G}$ satisfies the {\em reduced} structure conditions
$$
|z|^p \,-\, g_0(x) \;\leq\; \mathcal{F}_0(x,u,z) \;\leq\; L|z|^p + g_0(x)
$$
for all $x \in B$, Lemma \ref{lemma_caccioppoli} therefore implies the Caccioppoli-type inequality
\begin{equation} \label{eq_caccioppoli-noh}
\int_{B_r} g_{(u-h)_+}^p\,d\mu \;\leq\;
\frac{C}{(R-r)^p} \int_{B_R} (u-h)_+^p\,d\mu \;+\; \|g_0\|_s \mu\big(A_R(h)\big)^{1-\frac{p}{Q}+\d}
\end{equation}
for concentric balls $B_{r}\subset B_{R}\subset B$.
\begin{defn} \label{defn_DG}
Let $\d > 0$ and $C, \gamma \geq 0$ be given.  A function $u \in N^{1,p}_{\rm loc}(\Omega)$ is in the class
$DG^+(\Om)=DG^+_p(\Omega;C,\gamma,\d)$ if the inequality
\begin{equation} \label{eq_DG}
\int_{B(x,r)} g_{(u-k)_+}^p d\mu \;\leq\;
\frac{C}{(R - r)^p} \int_{B(x,R)} (u-k)_+^p d\mu \,+\,
\gamma^p\, \mu(A_R(k))^{1 - \frac{p}{Q} + \d}.
\end{equation}
holds for all $k \in \R$ and all balls $B(x_0,r)$ and $B(x_0,R)$ in $\Omega$ with $0 < r < R$.  We say that $u$ is in the class $DG_p^-(\Omega,)$ if $-u$ is  in the class $DG_p^+(\Omega)$.  The \emph{De Giorgi class} on $\Omega$ with parameters $\delta$, $\gamma$, and $C$ is then the set of functions
$$
DG_p(\Omega) \;:=\; DG^+_p(\Omega) \cap
DG_p^-(\Omega).
$$
\end{defn}

\section{H\"older Continuity of Quasi-Minimizers} \label{sect_holder}

We now prove that functions in the De Giorgi class have H\"older continuous representatives.   This is a local property, so we may assume $\Om$ to be bounded. 

By adapting the proof of Lemma \ref{lemma_weakharnack}, one obtains estimates of the oscillation of $u \in DG_p(\Omega)$ on balls.  This observation, formulated below, will play a crucial step towards continuity (Theorem \ref{thm_holder}).

\begin{lemma} \label{lemma_iteration}
Let $B = B(x,R)$ be a ball in $\Omega$, let $u \in DG_p(\Omega)$, put
$$
M \;:=\; \sup_B u \; \textrm{ and } \;
m \;:=\; \inf_Bu  
$$
There exists $\e_0 = \e_0(p,Q,M,\d) > 0$ so that
\begin{enumerate}
\item if 
$\mu\big(A_R(M -\xi\osc_{B} u) \big) \leq \e_0 \, \mu(B)$
holds for $\xi>0$, then
\begin{eqnarray}\label{eq_firstalt}
{\rm either}\quad  \; u &\leq&
M- \frac{\xi}{2}\osc_{B} u \; \;\quad \mu\textrm{-a.e.\ on } \frac{1}{2}B  \\
{\rm or} \quad \; \osc_{B} u&\leq&
\xi^{-1} c \, R^{(Q'\d)/p}\notag
\end{eqnarray}
\item if 
$\mu\big(D_R(m +\xi\osc_{B} u) \big) \leq \e_0 \, \mu(B)$
holds for $\xi>0$, then
\begin{eqnarray}\label{eq_firstalt2}
{\rm either}\quad \; u &\geq &
m + \frac{\xi}{2}\osc_{B} u \; \; \quad \mu\textrm{-a.e.\ on } \frac{1}{2}B \\
{\rm or}\quad  \; \osc_{B} u&\leq&
\xi^{-1} c \, R^{(Q'\d)/p}\notag.
\end{eqnarray}
\end{enumerate}
\end{lemma}
For the homogeneous case $f_{1}=f_{0}=0$, the proof below shows that only the first alternatives \eqref{eq_firstalt} and \eqref{eq_firstalt2} occur.

\begin{proof}
As a shorthand, write $\omega :=\osc_{B} u$.  The argument is symmetric, so we prove the first case only. Consider levels
$$
k_n \;:=\; \Big(M- \frac{\xi \omega}{2}\Big) - 2^{-n}\Big(\frac{\xi \omega}{2}\Big),
$$
let $B_n$ be the same sequence of balls centered at $x$ as before,
$$
B_n \;=\; B(x,r_n), \trm{ for } r_n=\frac{R}{2}(1+2^{-n})
$$
and consider the sequence of integrals
$$
Y_n \;:=\; \frac{1}{k_0^p} \dashint_{B_n} (u-k_n)_+^p\,d\mu.
$$
Following the proof of Lemma \ref{lemma_weakharnack}, we obtain an inequality similar to \eqref{eq_fundamental_estimate}:
\[
\begin{split}
\frac{2^{np}k_0^p}{(\xi \omega)^p} Y_{n+1} \leq\;&
2^{np}C k_0^{p(1+\frac{p}{Q})} \Big[\frac{2^{np}}{(\xi \omega)^p} Y_n\Big]^{1+\frac{p}{Q}} \\ & +
\frac{2^{np}Ck_0^{p(1+\delta)}  \gamma(1+k_{n+1}^p)\mu(B_{n+1})^\delta}{(\xi \omega)^p}
\Big[ \frac{2^{np}}{(\xi \omega)^p} Y_n\Big]^{1+\delta}.
\end{split}
\]
Now suppose that the second conclusion fails, so that
$$
(\xi \omega)^p \;>\;
c^p \,R^{Q'\d} \,\geq\,
c^p \,\mu(B)^\d.
$$
Then the previous inequality takes the form
\[
\begin{split}
 Y_{n+1}\leq C \left(\frac{k_0}{\xi \omega}\right)^\frac{p^2}{Q} [2^{np} Y_n]^{1+\frac{p}{Q}} +
 C \gamma(1+k_{n+1}^p) \left(\frac{k_0}{\xi \omega}\right)^{p\d} [2^{np} Y_n]^{1+\delta}.
\end{split}
\]
Now with the parameters $\sigma$, $\sigma'$, and $b$ as in \eqref{eq_iter-param}, and with
\begin{eqnarray*}
C = c \max\left\{ \left( \frac{k_0}{\xi\omega} \right)^\frac{p^2}{Q},\; \left( \frac{k_0}{\xi\omega} \right)^{p\d} \right\},
\end{eqnarray*}
we obtain the iteration inequality
$$
Y_{n+1} \;\leq\; C b^n Y_n^{1+\sigma}.
$$
From our choice of levels $k_n$, we obtain $u - k_0 \,=\, u -M + \xi\omega \;\leq\; \xi\omega$.
This and the density condition imply that
\begin{eqnarray*}
Y_0 &=& \frac{1}{k_0^p\mu(B_0)} \int_{A_{R_0}(k_0)} (u-k_0)_+^p\,d\mu \;\leq\;
\frac{1}{k_0^p \mu(B)} \int_{A_R(M - \xi\omega)} (\xi\omega)^p \,d\mu \\ &=&
\frac{(\xi\omega)^p}{k_0^p} \, \frac{\mu(A_R(M-\xi\omega))}{\mu(B)} \;\leq\;
\e_0 \left( \frac{\xi\omega}{k_0} \right)^p.
\end{eqnarray*}
By the previous calculation, choosing $\e_0 > 0$ sufficiently small, it follows that
\[
\begin{split}
Y_0 \;\leq\; \e_0 \left( \frac{\xi\omega}{k_0} \right)^p
&\leq\; b^{-1/\sigma^2}C^{-1/\sigma}\\
&=\; b^{-1/\sigma^2}c \min\left\{ \left( \frac{\xi\omega}{k_0} \right)^\frac{p^2}{Q},\; \left( \frac{\xi\omega}{k_0} \right)^{p\d} \right\}^{1/\min\{\frac{p}{Q},\d\}}.
\end{split}
\]
So by Lemma \ref{lemma_fastconv}, we obtain the convergence
$$
0 \;=\;
\lim_{n \to \infty} Y_n \;=\;
\frac{1}{k_0^p} \dashint_{\frac{1}{2}B} \Big(u - \Big(M - \frac{\xi\omega}{2} \Big) \Big)_+^p \,d\mu
$$
as well as an upper bound for $u$ on $\frac{1}{2}B$:
$$
\hspace{10em} u \;\leq\; M - \frac{\xi\omega}{2} \; \mu\trm{-a.e.\ on }\; \frac{1}{2}B. \hspace{10em} \qedhere
$$
\end{proof}

We recall two facts.  The first is a direct analogue of \cite[Eq 5.1]{KinnunenShan}, which replaces the role of the ``discrete isoperimetric inequality'' in $\R^n$ \cite{DeGiorgi}.  Apart from differences between Definition \ref{defn_DG} and \cite[Defn 3.1]{KinnunenShan} and the constants in \eqref{eq_DG} versus
\cite[Eq.\ 3.1]{KinnunenShan}, the proof is identical.  Below, $\Lambda $ refers to the constant from Lemma \ref{lemma_SPI}. The proof  uses Poincar\'e's and H\"older's inequalities, together with the fact that $u\in DG_p(\Om)$. 

\begin{lemma} \label{lemma_discreteiso}
Let $u \in DG_p(\Omega)$ and let $h < k$.  If $B = B(z,R)$ is a ball in
$X$ so that $2\Lambda B \subset \Omega$ and so that, for some $\theta \in
(0,1)$, the density condition
$$
\mu(A_R(h)) \;\leq\;
\theta \, \mu(B)
$$
holds, then there exists $c = c(\gamma, p, Q, \Lambda) > 0$ such
that, for all $q \in (1,p)$,
\[
\begin{split}
\mu(A_R(k)) \;\leq\;&\; 
\frac{c \, \mu(B)^{1-\frac{1}{q}}}{k-h} \,
\Big(
\mu(A_{\Lambda R}(h)) - \mu(A_{\Lambda R}(k))
\Big)^{\frac{1}{q}-\frac{1}{p}} \\ & 
\cdot \Big(
\int_{B_{2\Lambda R}} (u-h)_+^p \,d\mu \;+\;
\gamma^p R^p \mu(A_{2\Lambda R}(h))^{1-\frac{p}{Q} + \d}
\Big)^\frac{1}{p}.
\end{split}
\]
\end{lemma}

For functions $u \in DG_p(\Omega)$, we now consider the measure decay properties of their super-level sets.  The lemma below is proved by a standard telescoping argument; see also \cite[Lemma 5.2]{KinnunenShan} and \cite[Prop 10.5.1]{DiBenedetto}.
As a shorthand, we write
$$
M := \sup_{2\Lambda B}u \; \textrm{ and } \;
m := \inf_{2\Lambda B}u \;
$$
which are well-defined parameters, by Lemma \ref{lemma_weakharnack}.

\begin{lemma} \label{lemma_getdensity}
Let $u \in DG_p(\Omega)$ and let $B = B(z,R)$ be a ball in $X$ so that $2\Lambda B \subset \Omega$.
\begin{enumerate}
\item If there exists $\theta \in (0,1)$ such that the density condition
$$
\mu\Big(A_R\Big(M-\frac{1}{2}\osc_{2\Lambda B}u\Big)\Big) \leq \theta \, \mu(B)
$$
holds, then for each $\e > 0$, there exists $\xi \in (0,1)$ so that
\begin{eqnarray} \label{eq_firstalt3}
\textrm{either } \;
\e \, \mu(B) &\geq& \mu\Big(A_R\Big(M-\xi\osc_{2\Lambda B}u\Big)\Big) \\
\textrm{ or } \;
\osc_{2\Lambda B}u &\leq& \xi^{-1}\gamma^p R^{(Q'\d)/p}. \notag
\end{eqnarray}
\item If there exists $\theta \in (0,1)$ such that the density condition
$$
\mu\Big(D_R\Big(m + \frac{1}{2}\osc_{2\Lambda B}u\Big)\Big) \leq \theta \, \mu(B)
$$
holds, then for each $\e > 0$, there exists $\xi \in (0,1)$ so that
\begin{eqnarray} \label{eq_firstalt4}
\textrm{either } \;
\e \, \mu(B) &\geq& \mu\Big(D_R\Big(m+\xi\osc_{2\Lambda B}u\Big)\Big) \\
\textrm{ or } \;
\osc_{2\Lambda B}u &\leq& \xi^{-1}\gamma^p R^{(Q'\d)/p}. \notag
\end{eqnarray}
\end{enumerate}
\end{lemma}

Similarly as in Lemma \ref{lemma_iteration}, only the first alternatives \eqref{eq_firstalt3} and \eqref{eq_firstalt4} occur for the homogeneous case $f_0=f_1=0$.

\begin{proof}
The argument is symmetric, so we prove the first case only.  As a shorthand, let $\omega := \osc_{2\Lambda B}u$.  Consider levels of the form
$$
k_n \;:=\; M - 2^{-n}\omega.
$$
Observe that $k_n \to M$ as $n \to \infty$ and that
$$
M - k_n \;=\; 2^{-n}\omega \;=\; \frac{1}{2}(k_{n+1} - k_n).
$$
Put $\xi := 2^{-N}$, for some $N \in \mathbb{N}$ to be chosen later.
Now suppose the second conclusion fails.
Then for each $n = 1, 2, \ldots, N$, we have 
\begin{equation} \label{eq_notball}
\omega \;>\;
\xi^{-1}\gamma^p R^{(Q'\d)/p} \;\geq\;
C 2^n\gamma^p\mu(B)^\frac{\d}{p}.
\end{equation}
Using the density condition hypothesis with $\theta$, we now apply Lemma \ref{lemma_discreteiso} with $h= k_n $ and $k=k_{n+1}$, for each $n \in \N$, 
to obtain
\begin{eqnarray*}
\mu(A_R(k_{n+1})) &\leq&
\frac{c\mu(B)^{1-\frac{1}{q}}}{k_{n+1}- k_n } \,
\Big(
\mu(A_{\Lambda R}(k_n)) - \mu(A_{\Lambda R}( k_{n+1} ))
\Big)^{\frac{1}{q}-\frac{1}{p}} \\ & & \;
\cdot \Big(
\int_{2\Lambda B} (u-k_n)_+^p \,d\mu \;+\;
\gamma^p R^p \mu(A_{2\Lambda R}(k_n))^{1-\frac{p}{Q} + \d}
\Big)^\frac{1}{p}.
\end{eqnarray*}
From this and  $u - k_{n} \leq M - k_{n}\leq 2^{-n} \omega$, it follows that
\begin{equation} \label{eq_telescope}
\begin{split}
\omega \, \mu(A_R(k_{n+1})) \;\leq\;&\;
c2^{n+1} \, \mu(B)^{1-\frac{1}{q}} \, \big(
\mu(A_{\Lambda R}(k_{n})) - \mu(A_{\Lambda R}( k_{n+1} ))
\big)^{\frac{1}{q}-\frac{1}{p}} \\ &\;
\cdot \Big(
\frac{\omega^p \, \mu(2\Lambda B)}{2^{np}} \;+\;
\gamma^p R^p \mu(2\Lambda B)^{1-\frac{p}{Q} + \d}
\Big)^\frac{1}{p}.
\end{split}
\end{equation}
According to Lemma \ref{lemma_lowerahlforsreg}, we have  $C R^Q \leq \mu(2\Lambda B)$ for $C>0$.  From this and  \eqref{eq_notball}, for $n=N$, we further estimate
\begin{eqnarray*}
R^p \, \mu(2\Lambda B)^{1-\frac{p}{Q} + \d} &=&
\Big( (R^Q)^p \, \mu(2\Lambda B)^{Q-p + Q\d} \Big)^\frac{1}{Q} \\ &\leq&
\Big( \Big(\frac{\mu(2\Lambda B)}{C}\Big)^p \mu(2\Lambda B)^{Q -p + Q\d} \Big)^\frac{1}{Q} \\ &\leq&
C^{-\frac{p}{Q}} \, \mu(2\Lambda B)^{1+\d} \;\leq\;
C\frac{\omega^p}{2^{np}} \, \mu(2\Lambda B).
\end{eqnarray*}
Equation \eqref{eq_telescope} therefore becomes
\[
\begin{split}
\omega \, \mu(A_R&(k_{n+1})) \\
\;\leq\;&
C2^{n+1} \, \mu(B)^{1-\frac{1}{q}} \,
\Big( \frac{\omega^p \, \mu(2\Lambda B)}{2^{np}} \Big)^\frac{1}{p}
\cdot \Big(
\mu(A_{\Lambda R}(k_{n})) - \mu(A_{\Lambda R}( k_{n+1 } ))
\Big)^{\frac{1}{q}-\frac{1}{p}}
\\ \;\leq\;&
C \, \omega \, \mu(2\Lambda B)^{1-\frac{1}{q}+\frac{1}{p}} \,
\Big(
\mu(A_{\Lambda R}(k_n)) - \mu(A_{\Lambda R}( k_{n+1} ))
\Big)^{\frac{1}{q}-\frac{1}{p}}
\end{split}
\]
and therefore we have
\begin{eqnarray*}
\mu(A_R(k_{N+1}))^\frac{pq}{p-q} &\leq&
\mu(A_R(k_{n+1}))^\frac{pq}{p-q} \\ &\leq&
C \, \mu(2\Lambda B)^{\frac{pq}{p-q}-1} \,
\Big(
\mu(A_{\Lambda R}(k_n)) - \mu(A_{\Lambda R}( k_{n+1} ))
\Big).
\end{eqnarray*}
For $N \in \N$, we sum over the previous inequality and obtain
\begin{eqnarray*}
\mu(A_R(k_{N+1}))^\frac{pq}{p-q} &\leq&
\frac{C}{N} \, \sum_{n=0}^N \mu(2\Lambda B)^{\frac{pq}{p-q}-1} \,
\big(
\mu(A_{\Lambda R}(k_n)) - \mu(A_{\Lambda R}( k_{n+1} ))
\big) \\ &\leq&
\frac{C}{N} \, \mu(2\Lambda B)^{\frac{pq}{p-q}-1} \,
\big(
\mu(A_{\Lambda R}(k_0)) - \mu(A_{\Lambda R}(k_{N+1}))
\big) \\ &\leq&
\frac{C}{N} \, \mu(2\Lambda B)^\frac{pq}{p-q}.
\end{eqnarray*}
With $\e > 0$, choose $N \in \N$ so that $C \leq N\e^\frac{pq}{p-q}$.  From our previous choices of $\xi = 2^{-N}$ and $k_{N+1} = M - \xi\omega$, we obtain the first conclusion
$$
\hspace{5em} \mu(A_R(M-\xi\omega)) \;=\;
\mu(A_R(k_{N+1})) \;\leq\;
\e \mu(2\Lambda B). \hspace{5em}  \qedhere
$$
\end{proof}

Given a function in $DG_p(\Omega)$, we now prescribe its modulus of 
continuity from the density of its level sets.  We first explain the idea.

We estimate the oscillation of $u$ in two stages.  By a trivial estimate, either the sub- or the super-level set of $u$ has density at most $\tfrac{1}{2}$.  After applying Lemma \ref{lemma_getdensity}, we see that either the oscillation is already bounded, or the sub- or super-level set has even smaller density.  If the second alternative occurs, then we apply Lemma \ref{lemma_iteration}, so either the oscillation is already bounded, or we obtain a pointwise bound for $u$, as desired.
As before, we assume that Standard Hypotheses \ref{standhyp} are in force.

\begin{thm} \label{thm_holder} \label{cor_holder}
There exist $C > 1$, $\alpha > 0$, depending only on the parameters, so that for all $u \in DG_p(\Omega)$ and all balls $B(x,r) \subset B(x,R) \subset \Omega$, we have
\begin{equation} \label{eq_modcty}
\mathop{\osc}_{B(x,r)} u \;\leq\;
C \, \max\Big\{
\big(\mathop{\osc}_{B(x,R)} u\big)
\Big(\frac{r}{R}\Big)^\alpha, \;
\gamma^p r^{(Q'\d)/p}
\Big\}.
\end{equation}
In particular, every $u \in DG_p(\Omega)$ has an a.e.\ representative that is locally $\beta$-H\"older continuous, with
$\beta := \min\{ \alpha, (Q'\d)/p \}$.
\end{thm}

\begin{proof}
As before, let $M$ and $m$ be the supremum and infimum of $u$ on $ 2\Lambda B$, respectively,
and let $\omega := \osc_{2 \Lambda ' B} u$.
We observe that
$$
\frac{M+m}{2} \;=\; M - \frac{\omega}{2} \;=\; m + \frac{\omega}{2}.
$$
So for $\theta = \frac{1}{2}$, one of the inequalities
$$
\mu\left(D_R\left(M - \frac{\omega}{2}\right)\right) \leq \theta \mu(B(x,R))
\, \textrm{ or } \,
\mu\left(A_R\left(m + \frac{\omega}{2}\right)\right) \leq \theta \mu(B(x,R))
$$
must hold.
The argument is symmetric, so suppose the rightmost inequality holds.  Lemma \ref{lemma_getdensity} implies that for each $\e > 0$, there exists $\xi > 0$ satisfying 
$$
\omega \;\leq\; \xi^{-1} \gamma^p\mu(B(x,R))^\frac{\d}{p}
\; \textrm{ or } \;
\mu\big(A_R(M-\xi\omega)\big) \;\leq\; \e \, \mu(B(x,R)).
$$
If the leftmost inequality holds, then $\omega$ is bounded.
Suppose instead that the rightmost inequality holds. Applying Lemma \ref{lemma_iteration}, there exists $\e_0 > 0$, depending only on $p, Q, M, \d$ such that each $\xi$ satisfying the rightmost inequality, with $\e=\e_0$, either satisfies the estimate
\begin{equation}
\label{eq_u_bound}
\begin{split}
u \;\leq\; M - \frac{\xi \omega}{2}
\end{split}
\end{equation}
$\mu$-a.e.\ on $B(x,R/2)$ or $\omega$ is again bounded.  Equation \eqref{eq_u_bound} and the elementary inequality $(-\inf_{B(x,R)} u) \leq (-m)$ imply that
$$
\mathop{\osc}_{B(x,R/2)}u \;:=\; 
\sup_{B(x,R/2)}u - \inf_{B(x,R/2)}u \;\leq\; 
M- \frac{\xi\omega}{2} - m  \;=\;
\lambda\, \omega
$$
where $\lambda := 1 - \frac{\xi}{2}$.  Replacing $2\Lambda R$ by $r_{n+1}$ and $R/2$ by $r_n$, where $r_n := (4\Lambda )^{-n}R$, we iterate the argument to obtain
\begin{eqnarray*}
\mathop{\osc}_{B(x, r_{n+1})}u &\leq&
\max\left\{ \lambda\big(\mathop{\osc}_{B(x, r_n)}u\big),\, \xi^{-1}\gamma^p r^\frac{Q\d}{p} \right\} \\ &\leq&
\max\left\{ \lambda^n\big(\mathop{\osc}_{B(y,R)}u\big),\, \xi^{-1}\gamma^p r^\frac{Q\d}{p} \right\}.
\end{eqnarray*}
Equation \eqref{eq_modcty} follows, where $\alpha$ solves $\lambda^n = \big(r_n/R\big)^\alpha  = (4\Lambda )^{-n\alpha}$.
\end{proof}

\section{Harnack Inequalities for Quasi-Minimizers} \label{sect_harnack}

As a consequence of H\"older continuity, we prove a Harnack-type inequality for quasi-minimizers.  For the homogeneous case \cite[Sect 7]{KinnunenShan}, the proof of the Harnack inequality uses a covering argument in the spirit of Krylov and Safonov \cite{KrylovSafonov}.  We note that a variant of the argument is also valid in our setting.

Our approach follows the ``expansion of positivity'' technique \cite{DiBenedetto1} instead, which relies on iteration techniques as in the previous sections; see also \cite{DiBenedetto}.  We begin with a version of the density theorems from previous sections.

\begin{lemma}[Expansion of positivity] \label{lemma_expansion}
If $u \in DG_p(\Omega)$ with $u>0$, and if $h > 0$ satisfies the density condition
$$
\mu\big(A_R(h)\big) \;\geq \; \frac{1}{2}\mu(B(x,R)),
$$
then there exists $\xi\in(0,1)$ so that
\begin{eqnarray*}
\textrm{either } \quad h &\leq& \xi^{-1} \gamma \, R^{(Q'\d)/p} \\
\textrm{ or } \qquad u &\geq & \xi h \;\qquad  \mu\textrm{-a.e.\ on } \; B(x,2R).
\end{eqnarray*}
\end{lemma}

\begin{proof}
Combining the doubling condition and the above hypothesis, we obtain
$$
\mu(B(x,4R)) \;\leq\;
c_\mu^2 \, \mu(B(x,R)) \;\leq\;
2c_\mu^2 \, \mu\big(A_R(h)\big) \;\leq\;
2c_\mu^2 \, \mu\big(A_{4R}( h)\big).
$$
Putting $\theta = 1 - \frac{1}{2}c_\mu^{-2}$, we further obtain the density condition
$$
\mu\big(D_{4R}(h)\big) =\mu(B(x,4R))-\mu\big(A_{4R}( h)\big)\;\leq\;
\theta \, \mu(B(x,4R)).
$$
Observe that the proof of Lemma~\ref{lemma_getdensity} uses $m$ and $\omega$ only as numerical parameters. We therefore use a similar argument with $m=0$ and $\omega=2h$ and with $B(x,4R)$ in place of $B(x,R)$. This implies that
for each $\e >0$ there exists $\xi \in (0,1)$ so that
\begin{eqnarray*}
\textrm{either } \quad \qquad h &\leq&
\xi^{-1} \gamma\, R^{(Q'\d)/p} \\
\textrm{or } \quad
\e \mu(B_R) &\leq& \mu(D_R(2\xi h)).
\end{eqnarray*}
Similarly the proof of Lemma~\ref{lemma_iteration} remains valid under the same change of parameters, thus completing the proof.
\end{proof}

We now arrive at the Harnack inequality, and the proof is in two parts.  In Part (1) we use H\"older continuity to obtain an initial density estimate for $u \in DG_p(\Omega)$ in a smaller ball.  In Part (2) the density estimate allows us to iterate Lemma \ref{lemma_expansion} to prove Harnack's inequality and expand its validity to the original ball.  One technical difficulty is that the constants in the inequality are increasing with each iteration.  To overcome this, we choose the radius of the smaller ball, and thus the number of iterations, according to the supremum.  To make our choices explicit, we use an auxiliary (radial) function.
\begin{thm} \label{thm_harnack}
Let $u \in DG_p(\Omega)$ with $u > 0$. Then there exist $C, c > 0$, depending
only on Standing Hypotheses \ref{standhyp} and the parameters in $DG_p(\Om)$  so that 
$$
\sup_{B} u  \;\leq\;
C\inf_{B} u \;+\; c\, R^{(Q'\d)/p}
$$
for all balls $B = B(x_0,R_0)$ in $X$ so that $4B \subset \Omega$.
\end{thm}

\begin{proof}
{\bf Part (1): Pointwise estimates}.
Let $u \in DG_p(\Omega;C,\gamma,\d)$ be given.  For each $x\in B(x_0,R_0)$, consider the function
$$
v \;=\; \frac{u}{u(x)}.
$$
Clearly we have $v \in DG_p(\Omega;C,\Gamma,\d)$, where $\Gamma := 
\big(u(x)\big)^{-1} \gamma$.  Next, define
$$
M(r) \;:=\; \sup_{B(x,r)}v \; \quad \trm{and}\quad  \;
N(r) \;:=\; (1 - r/R_0)^{-\beta}
$$
with 
$\beta > 0$ to be chosen later.  Since $v$ is continuous (Theorem \ref{thm_holder}), we have the identity $M(0)=1= N(0)$ as well as the inequality
$$
\lim_{r \nearrow R_0} M(r) \;<\; \infty \;=\;
\lim_{r \nearrow R_0} N(r),
$$
so there must be a largest root $r_0 > 0$ of the equation $M(r)=N(r)$. The advantage of using the auxiliary function $N(r)$ is that it gives an explicit dependence between the radius and the supremum. This is useful in \eqref{eq_independent_bound} where, after fixing $\beta$, we see that the constant remains under control in iteration.

To continue, there exists $y_0 \in \overline{B}(x,r_0)$ at which $v$ attains the supremum
\begin{equation}
\label{eq_N=M}
\begin{split}
v(y_0) \;=\; \sup_{B(x,r_0)}v= M(r_0) \;=\; N(r_0) \;=\; (1 - r_0/R_0)^{-\beta}.
\end{split}
\end{equation}
For
$$
R := \frac{R_0-r_0}{2},
$$ the triangle inequality gives
$$
d(x,y_0) + R \;\leq\;
r_0+\frac{1}{2}(R_0-r_0) \;=\;
\frac{1}{2}(R_0+r_0),
$$
and thus, because  $r_0$ is the largest root, we obtain the estimate
\begin{eqnarray*}
\sup_{B(y_0,R)} v \;\leq\;
\sup_{B\left(x,\frac{R_0+r_0}{2}\right)} v &=&
M\Big(\frac{R_0+r_0}{2}\Big) \;\leq\;
N\Big(\frac{R_0+r_0}{2}\Big) \\ &=&
\Big(\frac{R_0-r_0}{2R_0}\Big)^{-\beta} \;=\;
{2^\beta}N(r_0).
\end{eqnarray*}
Applying Theorem \ref{thm_holder} again together with the above inequality, we have, for each $ \rho  \in (0,R)$ and each $y \in B(y_0, \rho )$,
\[
\begin{split}
v(y) - v(y_0) \;\geq\;
-\osc_{B(y_0, \rho )} v \;\geq\;&
-C\Big[ \big( \sup_{B(y_0,R)}v \,-\, \inf_{B(y_0,R)}v\big) \Big(\frac{ \rho }{R}\Big)^\alpha \,+\,  \rho ^{(Q'\d)/p}\Big] \\
\;\geq\;& -C\Big[
{2^\beta}N(r_0)\Big(\frac{ \rho }{R}\Big)^\alpha \,+\,  \rho ^{(Q'\d)/p}\Big].
\end{split}
\]
We now set $ \rho  := \e R$, and choose $\e  > 0$ sufficiently small, so that
$$
C\Big[ {2^\beta}N(r_0)\e^\alpha \,+\, (\e R)^{(Q'\d)/p} \Big] \;\leq\;
\frac{1}{2}N(r_0),
$$
where we estimated
\[
\begin{split}
(\e R)^{(Q'\d)/p}=\left(\frac{\e(R_0-r_0)}{2}\right)^{(Q'\d)/p} \;\leq\;
\frac{1}{4} \left(\frac{R_0-r_0}{R_0}\right)^{-\beta}= \frac{1}{4}N(r_0)
\end{split}
\]
for the second term on the left hand side. Notice that $\e$ depends on $\beta$, but can be chosen independently of $r_0$.
This together with \eqref{eq_N=M} implies
\[
\begin{split}
v(y) - v(y_0) \;\geq\; &
-C\Big[ {2^\beta}N(r_0)\Big(\frac{ \rho }{R}\Big)^\alpha \,+\,  \rho ^{(Q'\d)/p} \Big]
\\ \;\geq\;&
-C\Big[ {2^\beta}N(r_0)\e^\alpha \,+\, (\e R)^{(Q'\d)/p} \Big]
\;\geq\;
-\frac{1}{2}N(r_0) \;=\; -\frac{1}{2}v(y_0),
\end{split}
\]
which further implies the pointwise estimate
$$
v(y) \;\geq\; \frac{1}{2}v(y_0) \;=\; \frac{1}{2}\sup_{B(x,r_0)}v \;=:\; h
$$
for $\mu$-a.e.\ $y \in B(y_0, \rho )$.
This implies the density condition
\[
\begin{split}
\mu(A_\rho(h)) \;\geq\; \half  \mu(B(y_0,\rho)).
\end{split}
\]

{\bf Part (2): Expansion of positivity}.
We now apply Lemma~\ref{lemma_expansion}, so there exists a constant $\xi \in (0,1)$, depending only on the parameters of Standing Hypotheses \ref{standhyp} and Structure Conditions \ref{structcond}, such that
\begin{eqnarray*}
\textrm{either } \quad h &\leq& \xi^{-1} \Gamma \,  \rho ^{(Q'\d)/p} \\
\textrm{ or } \qquad v &\geq & \xi h \;\qquad  \mu\textrm{-a.e.\ on } \; B(y_0,2\rho).
\end{eqnarray*}
The second inequality implies the modified density condition
\[
\begin{split}
\mu(A_{2\rho}(\xi h))\geq \half  \mu(B(y_0,2\rho)),
\end{split}
\]
and thus we can iterate Lemma~\ref{lemma_expansion}. If the first alternative occurs, we get the desired bound, and if the second alternative occurs for $n-1$ times,  we have
\begin{eqnarray*}
\textrm{either } \quad \xi^nh  &\leq& \Gamma \,  (2^n\rho) ^{(Q'\d)/p} \\
\textrm{ or } \qquad v &\geq & \xi^n h \;\qquad  \mu\textrm{-a.e.\ on } \; B(y_0,2^n\rho)
\end{eqnarray*}
on the $n$th round.
For sufficiently large $n$, we have $B(x,4R_0)\subset B(y_0,2^n\rho)$.  In either case, we obtain
\[
\begin{split}
\xi^n h &\;\leq\;
\max\left\{\inf_{B(x_0, 4R_0)} v,\, \Gamma \,  (2^n\rho) ^{(Q'\d)/p}\right\} \\ & \;\leq\;
\max\left\{\inf_{B(x_0, R_0)} v,\, \Gamma \,  (2^n\rho) ^{(Q'\d)/p}\right\}.
\end{split}
\]
Finally, we estimate $\xi^nh$ from below by a constant depending only on data by utilizing the auxiliary function.  First, we choose $n \in \N$ so that
\[
\begin{split}
2^{n-1}\rho \;\leq\; 4 R_0 \leq 2^n \rho \;=\; 2^n \eps \frac{R_0-r_0}{2}
\end{split}
\]
so that
\[
\begin{split}
\frac{8R_0}{\eps (R_0-r_0)} \leq 2^n.
\end{split}
\]
We now choose $\beta$ so that $\xi 2^\beta=1$, from which it follows that
\begin{equation}
\label{eq_independent_bound}
\begin{split}
\xi^nh \;=\;
2^{-\beta n}h \;\geq\;
\left(\frac{8 R_0}{\eps (R_0-r_0)}\right)^{-\beta}\half (1-\frac{r_0}{R_0})^{-\beta} \;=\;
2^{3\beta-1}\eps^{\beta} \;=:\; C
\end{split}
\end{equation}
and therefore we obtain the estimate
\begin{eqnarray*}
C \;\leq\; \xi^n h &\leq&
\max\left\{\inf_{B(x_0, R_0)} v,\, \Gamma (2^n\rho) ^{(Q'\d)/p}\right\} \\ &\leq&
\max\left\{\inf_{B(x_0, R_0)} \frac{u}{u(x)},\, \frac{\gamma(2R_0) ^{(Q'\d)/p}}{u(x)} \right\} \\
C u(x) &\leq&
\max\left\{\inf_{B(x_0, R_0)} u,\, \gamma2^{(Q'\d)/p}R_0^{(Q'\d)/p}\right\}
\end{eqnarray*}
Taking suprema over all $x \in B$, the theorem follows.
\end{proof}

\bibliographystyle{alpha}
\bibliography{degiorgi-final}
\end{document}